\documentclass[a4paper]{amsart}

\usepackage[all]{xy}
\usepackage{a4wide}
\usepackage{verbatim}
\usepackage{amsmath}
\usepackage{amssymb}
\usepackage{amsthm}
\usepackage{latexsym}
\usepackage{enumerate}
\usepackage{xcolor}
\usepackage{prettyref}
\usepackage{graphicx}
\usepackage[pdfencoding=auto]{hyperref}
\usepackage{bm}
\usepackage[linewidth=1pt]{mdframed}
\usepackage{mathrsfs}
\usepackage[lite]{amsrefs}
\usepackage[linewidth=1pt]{mdframed}
\usepackage{tikz}
\usetikzlibrary{matrix,arrows}
\usepackage{tikz-cd}
\usepackage{float}

  \newcommand{\rationals}{{\mathbb{Q}}}
\newcommand{\Tits}{\mathcal{T}}
\newcommand{\St}{\operatorname{St}}

\newtheorem{theorem}{Theorem}[section]
\newrefformat{thm}{\hyperref[{#1}]{Theorem~\ref*{#1}}}
\newtheorem{definition}[theorem]{Definition}
\newrefformat{def}{\hyperref[{#1}]{Definition~\ref*{#1}}}

\newrefformat{lem}{\hyperref[{#1}]{Lemma~\ref*{#1}}}

\newrefformat{prop}{\hyperref[{#1}]{Proposition~\ref*{#1}}}
\newtheorem{corollary}[theorem]{Corollary}
\newrefformat{cor}{\hyperref[{#1}]{Corollary~\ref*{#1}}}

\theoremstyle{definition}
\newtheorem*{remark}{Remark}
\newtheorem{recall}[theorem]{Recollection}

\newrefformat{rem}{\hyperref[{#1}]{Remark~\ref*{#1}}}

\newenvironment{proofof}[1]{\vspace{.2cm}\noindent\textsc{Proof of
    #1:}}{\hspace*{\fill} $\blacksquare$\par\vspace{.1cm}}

\newcommand*{\Homol}{\operatorname{H}}
\newcommand*{\SL}{\operatorname{SL}}
\newcommand*{\GL}{\operatorname{GL}}
\newcommand*{\PGL}{\operatorname{PGL}}
\newcommand*{\vcd}{\operatorname{vcd}}
\newcommand{\Z}{\mathbb{Z}}
\newcommand{\N}{\mathbb{N}}
\newcommand{\Q}{\mathbb{Q}}
\newcommand{\ringO}{\mathcal{O}}

\begin{document}

\title[On the Calegari-Venkatesh conjecture on algebraic K-Theory]{On the Calegari-Venkatesh conjecture \\connecting modular forms spaces and algebraic K-Theory}

\author{Alexander D. Rahm$^1$ and Emiliano Torti$^1$}

\date{\today}

\address{$^1$Laboratoire de math\'ematiques GAATI, Universit\'e de la Polyn\'esie Fran\c{c}aise, BP 6570, 98702 Faa'a, French Polynesia\\ Alexander.Rahm@upf.pf , ORCID: 0000-0002-5534-2716, https://gaati.org/rahm}


\begin{abstract}
Calegari and Venkatesh constructed, modulo small torsion, a surjective homomorphism from the degree-two homology of $\mathrm{PGL}_2$ over a ring of algebraic integers $\mathcal{O}$ (with odd class number and sufficiently many embeddings) onto the second Milnor $K$-group of formal symbols over $\mathcal{O}$. Motivated by their question of whether this result reflects a deeper connection between the homology of $\mathrm{PGL}_2$ and algebraic $K$-theory, we show how to use elements from Quillen's higher $K$-groups to construct, modulo small torsion, cohomological classes of $\mathrm{GL}_2$ in multiple degrees. This procedure allows us, in the framework of Calegari and Venkatesh, to construct (under extra hypotheses) non-trivial torsion classes in $\Homol_2(\mathrm{PGL}_2 (\mathcal{O}), \mathbb{Z})$ arising from elements in $K_2 (\mathcal{O})$. Our methods rely on the interplay between algebraic $K$-theory and Steinberg homology via Quillen's $Q$-construction. We conclude with some numerical examples.
\end{abstract}

\maketitle
\setcounter{tocdepth}{1}
\tableofcontents

\section{Introduction}
\noindent
Let $\ringO_F$ be the ring of integers in a number field $F$, let $r_1$ be the number of its real embeddings and $r_2$ the number of its conjugate pairs of complex embeddings. For each place $v$ of $F$, denote by $F_v$ the completion at $v$. Let $\mathbb{A}_F$ (or simply $\mathbb{A}$) be the ring of ad\`eles of $F$ and denote by $\mathbb{A}_f$ the subring of finite ad\`eles. Consider $G=\text{PGL}_2 / F$ as an algebraic group over $F$ and let $\Sigma$ be a finite set of finite places of $F$. Let $G_\infty := G (F \otimes_\mathbb{Q} \mathbb{R})$ and $K_\infty \subset G_\infty$ be a maximal compact subgroup. Let $K_{ \Sigma}$ be the compact open subgroup of $G(\mathbb{A}_f)$ efined by 
$K_{0, \Sigma}=\Pi_{v\in\Sigma} K_{0, v} \Pi_{v\not\in\Sigma} G(\mathcal{O}_{F_v })$, where $K_{0, v}$ are the images in $\text{PGL}_2 (F_v )$ of the following groups (we keep the same notation for simplicity):
$$K_{0, v} =\Big\{g \in \text{GL}_2 (\mathcal{O}_{F_v})\; |\; g \equiv \begin{pmatrix} * & * \\ 0 & *\end{pmatrix}  \mod \pi_v \Big\}.$$
If $K$ is a compact open subgroup of $G(\mathbb{A}_f)$, we can define the associated arithmetic quotient:
$$Y(K):= G(F) \backslash G(\mathbb{A}) / K_\infty K \cong G(F) \backslash (G(\mathbb{A}_f )/K\times G_\infty /K_\infty).$$
Denote $Y_0 (1)$ and $Y_0 (\mathfrak{p})$ (for a chosen prime $\mathfrak{p}$ in $\mathcal{O}_F$) the locally symmetric spaces obtained as above, corresponding respectively to the choices $K=\text{PGL}_2 (\mathcal{O}_\mathbb{A})$ and $K=K_{\Sigma}$ with $\Sigma=\{\mathfrak{p}\}$.\\
The spaces $Y(K_\Sigma)$ are endowed with an action of certain operators, called Hecke operators, obtained as double coset operators built from each element $g\in G(\mathbb{A}_f )$; and such operators preserve the (co)homology of the spaces $Y(K_\Sigma)$ (for more details see ~\cite{CalegariVenkatesh}*{sec. 3.4.3}). Let $\mathbb{T}_\Sigma$ be the subring of $\text{End} (\Homol_q (Y(K_\Sigma), \mathbb{Z}))$ (for a fixed positive integer $q$) generated by those Hecke operators. Let $I$ be the ideal of $\mathbb{T}_\Sigma$ generated by the elements of the form $T-\text{deg}(T)$ for all Hecke operators $T$. We called the ideal $I$ the Eisenstein ideal. We refer the reader to ~\cite{CalegariVenkatesh}*{sec. 3.8} for a discussion on alternative definitions of the Eisenstein Ideal. \\
Calegari and Venkatesh, in their remarkable work ~\cite{CalegariVenkatesh}, established (among many other results) an important connection between certain 2-homology classes of $Y_{0} (1)$ (or in some case certain 1-homology classes of $Y_{0} (\mathfrak{p})$ for any prime $\mathfrak{p}$ of $F$) and the $K$-theory group $K_2 (\mathcal{O})$. 
Let $\omega_F^{(2)}$ be the order of the group of Galois invariant elements in the 2nd Tate twist $\mathbb{Q}/\mathbb{Z} (2)$ (see Notation and Terminology for more details).\\
In particular, Calegari and Venkatesh proved the following (see ~\cite{CalegariVenkatesh}*{theorem 4.5.1}):
\begin{theorem}[Calegari-Venkatesh]
Let $R=\mathbb{Z}\big[\frac{1}{\omega_F^{(2)}}\big]$. If $F$ has at least two Archimedean places, i.e. $r_1 + r_2 \geq 2$, then there exists a surjective homomorphism:
$$\Homol_2 (Y_0 (1) , R ) / I \cdot \Homol_2 (Y_0 (1) , R ) \twoheadrightarrow K^M_2 (O_F )\otimes R.$$
\noindent If $F$ is a quadratic imaginary field, $\mathfrak{p}$ is any prime and assume that $\Homol_1 (Y_0 (1), \mathbb{C} )=0$. 
Then there exists a subgroup $\mathcal{K}_{\mathfrak{p}}$ inside $\Homol_1 ( Y_0 (\mathfrak{p}) , R ) $, and a surjective homomorphism: 
$$\mathcal{K}_\mathfrak{p} / I \cdot \mathcal{K}_\mathfrak{p} \twoheadrightarrow K^M_2 (O_F )\otimes R.$$
\end{theorem}
\noindent Many interesting questions arise from the above result (see for example ~\cite{CalegariVenkatesh}*{rem. 4.5.3}). For exmple, Calegari and Venkatesh ask if or when the above surjections are actually isomorphisms. More generally, they ask if such result represents a deeper connection between the homology of the relevant spaces and algebraic $K$-theory, i.e. if for example the other algebraic $K$-groups contribute in other homology degrees. \\
This work started with the motivation to understand, in the first part of Calegari and Venkatesh's result, where the connection between algebraic $K$-theory and the (co)homology of the space $Y_0 (1)$ might come from. First, we place ourself in the slightly simplified situation when the class number of $F$ is odd. Indeed, when the class number of $F$ is odd, the space $Y_0 (1)$ is a connected $K(\pi, 1)$-space, and $\pi_1 (Y_0 (1))\cong \text{PGL}_2 (\mathcal{O}_F )$ (see ~\cite{CalegariVenkatesh}*{rem. 3.6.1}). As a consequence, we can identify the homology of such space with the group homology of its fundamental group. \\ One way to read Calegari and Venkatesh's result consists in realizing that the existence of a surjective morphism implies the existence of a set-theoretical section, i.e. a way to construct from each element in $K_2 (\mathcal{O}_F )$ (modulo small torsion) an element in $H_2 (Y_0 (1) , R )$ modulo the Eisenstein ideal. The goal of this article is to propose of a way to construct, starting from elements in Quillen's higher $K$-group $K_N (\mathcal{O}_F )$ (for each fixed positive integer $N$), classes in the (co)homology of $PGL_2 (\mathcal{O}_F )$. In our opinion, it might be interesting to understand to what extent such construction can be used to produce an actual section for the Calegari-Venkatesh map when $N=2$. As an application of our main result, we show that under certain hypothesis on the field $F$, our construction reaches exactly the homology groups of Calegari and Venkatesh. Finally, we will discuss some numerical examples. \\
Our construction relies on the fact that we can use the Hurewicz homomorphism to connect K-theory to the Steinberg homology groups via the homology of the classifying spaces from Quillen's $Q$-construction with respect to the category of free modules and the relative rank filtration. We now introduce briefly all the relevant terminology and notation. We denote the Hurewicz homomorphism by $\text{Hu}_N : K_N (\mathcal{O}) \rightarrow \Homol_{N+1} (BQ , \mathbb{Z})$ for a fixed positive integer $N$, where $BQ_j$, with $j \in \N$, is the sequence of the classifying spaces from Quillen's $Q$-construction of algebraic $K$-Theory and $BQ=BQ_\infty$. Next, we introduce the homology with coefficients in the Steinberg module, although we shall later see that up to small torsion, we can translate it into ordinary cohomology.
To do so, we associate to $\GL_n(F)$ a \emph{Tits building} $\Tits_n(F)$.  The space $\Tits_n(F)$ is the $(n-2)$-dimensional
simplicial complex whose $p$-simplices are flags of subspaces
\[0 \subsetneq V_0 \subsetneq \cdots \subsetneq V_p \subsetneq F^n.\]
The group $\GL_n(F)$ acts on $\Tits_n(F)$ by simplicial automorphisms.
\begin{definition}
The \emph{Steinberg module} $\St_n(F)$
is the $\GL_n (F)$-module given by the reduced top homology group $\tilde{\Homol}_{n-2}(\Tits_n(F);\Z)$.
\end{definition}
\noindent
For our main result, we need to be sure to avoid some small torsion issues that could make some steps in our construction fail. There are essentially two distinct types of small torsion issues that could appear. The first one appears already in the work of Calegari and Venkatesh, and we make the following (see ~\cite{CalegariVenkatesh}*{sec. 3.3.6}):
\begin{definition}
Fixing $j\in \N$ and a number ring $\ringO$, we call \emph{orbifold primes} those prime numbers which occur as factors in the order of elements of $\GL_j(\ringO)$. 
\end{definition}
\noindent
Note that the orbifold primes are all dividing $\omega_F^{(2)}$. 
The second type of small torsion that could create problems concerns the higher $K$-groups (so strictly speaking not the framework of Calegari and Venkatesh). Indeed, since our construction will pass through the Hurewicz homomorphism, we need to ensure that we stay away from its kernel. Soul\'e (see ~\cite{Soule}*{prop. 3}) proved that every element in such kernel is torsion and later on, Arlettaz (see ~\cite{Arlettaz}*{corollary 1.6}) proved that it has actually finite exponent, i.e. it is a finite group. This brings us to the following:
\begin{definition}\label{arlettaz}
For every $N\in \mathbb{Z}_{\geq 1}$, let $R_{N-1}\in \mathbb{Z}_{\geq 1}$ be Arlettaz's universal bound for the torsion of the kernel of the Hurewicz homomorphism, 
i.e. $R_{N-1} \cdot \text{Ker}(\text{Hu}_N )=0$. A torsion element $g\in K_N (\mathcal{O})$ is called large torsion if its order does not divide $R_{N-1}$. 
\end{definition}
\noindent
Note that the positive integers $R_{N-1}$ are defined by an explicit recurrence relation (see ~\cite{Arlettaz}*{def. 1.3}) and they satisfy the useful property that a prime $p$ divides $R_{N-1}$ if and only if $p \leq (N+1) / 2$. In particular, each element of $K_2 (\mathcal{O})$ is large torsion and for example, we have $R_1 =1$, $R_2 =2$, $R_3 =4$, $R_4 =24$, $R_5 =144$, and $R_6 =288$.\\
Going back to Quillen's $Q$-construction, denote by $\pi_j : \Homol_{N+1} (BQ_j , \mathbb{Z}) \rightarrow \Homol_{N+1} (BQ , \mathbb{Z})$ and by $\phi_j : \Homol_{N+1}(BQ_j , \mathbb{Z}) \to \Homol_{N+1-j} (\GL_{j}(\mathcal{O}), \St_{j}(F))$ the homomorphisms given by the spectral sequence of Quillen's $Q$-construction (see ~\cite{Quillen}*{theorem 3}). \\
We are now ready to state the first result of this article: 
\begin{theorem}\label{lifting}
Let $\ringO$ be the ring of integers in a number field $F$ of class number one, and let $N$ be a positive integer. 
There exists a map
$$\varphi : K_N (\mathcal{O})\rightarrow \bigoplus\limits_{j=0}^N \Homol_{N+1-j} (\GL_{j}(\mathcal{O}), \St_{j}(F)),$$
such that: \\
(1) $\varphi$ factors via the following commutative diagram: 
$$
\begin{tikzcd}
\Homol_{N+1} (BQ , \mathbb{Z}) \arrow{r}{\psi}  & \bigoplus\limits_{j=0}^N \Homol_{N+1} (BQ_j , \mathbb{Z}) \arrow{d}{\phi:=\oplus \phi_j} \\
K_N (\mathcal{O})  \arrow{u}{\text{Hu}_N} \arrow{r}{\varphi} &  \bigoplus\limits_{j=0}^N \Homol_{N+1-j} (\GL_{j}(\mathcal{O}), \St_{j}(F))
\end{tikzcd}
$$
\noindent where $\text{Hu}_N$ and $\phi$ are the group homomorphisms introduced above;\\
\break
(2) let $\alpha \in K_N (\mathcal{O})$ be a large torsion element or of infinite order, we have 
$$\varphi(\alpha) = (0, \dots, 0, a_j , 0, \dots, 0)\not=0 \text{ for some }j,\text{ and}$$
$$\pi (\psi (\text{Hu}_N (\alpha)))= \text{Hu}_N (\alpha), $$
i.e. $\psi$ is a section of the group homomorphism $\pi=\oplus_j \pi_j$ wrt the image of the Hurewicz homomorphism. 
\end{theorem}
\noindent
Recall that the groups $K_N (\mathcal{O})$ are finite if $N$ is even and for each $N$ they are finitely generated ~(see Quillen's paper \cite{Quillen} for the finite generation, and Borel's proposition 12.2 \cite{Borel74} for the vanishing rank). 
We provide the proof of this theorem in Section~\ref{Steinberg_connection}. 
As a terminology remark, sometimes we refer to the images of elements in $K_N (\mathcal{O})$ via $\varphi$ as lifts (see Notation and Terminology for more details).
\begin{remark}
 We cannot, with the tools at hand in the present paper, produce a homomorphism from $K_N (\mathcal{O})$ to $ \oplus_{j=0}^{N} \Homol_{N+1-j} (\GL_{j}(\mathcal{O}), \St_{j}(F)).$ Indeed, for example, we are unable to guarantee that the lifts of elements of finite orders in $K_N (\mathcal{O})$ are are always of finite order.
\end{remark}

\begin{recall}[Bieri--Eckmann duality]\label{Bieri-Eckmann duality}
Recall that a theorem of Borel and Serre~\cite{BorelSerre}*{theorem 11.4.4} implies a formula~\cite{ChurchFarbPutman}*{equation (1.1)} for the virtual cohomological dimension (vcd) of the general linear group $\GL_n(\ringO)$, namely
$$\vcd(\GL_n(\ringO)) = r_1\frac{(n+1)n}{2}+r_2n^2-n.$$
We shall use the notation $\nu(n) = \vcd(\GL_n(\ringO))$.\\
The long exact sequence interweaving Steinberg homology, group cohomology and Farrell--Tate cohomology of $\GL_{j}(\ringO)$ (see \cite{Brown}) reduces, because Farrell--Tate cohomology is trivial at non-orbifold primes, to isomorphisms
$$\Homol_m(\GL_{j}(\ringO);\thinspace \St_j(F)) \cong \Homol^{\nu(j)-m}(\GL_{j}(\ringO);\thinspace \Z) \text{ modulo orbifold primes},$$
and hence from the target $\bigoplus\limits_{j=0}^{N} \Homol_{N+1-j} (\GL_{j}(\mathcal{O}), \St_{j}(F))$ of Theorem~\ref{lifting} we reach, modulo orbifold primes, the cohomology groups with trivial $\Z$-coefficients $\bigoplus\limits_{j=0}^{N}
\Homol^{\nu(N+1-j)-j}(\GL_{N+1-j}(\ringO))
.$
\end{recall}

\subsection{Application to Calegari and Venkatesh's connection between modular forms spaces and algebraic K-Theory}
Assuming that the class number of $\ringO$ is odd, and that $r_1+r_2 \geq 2$, the theorem of Calegari and Venkatesh~\cite{CalegariVenkatesh}*{theorem 4.5.1} introduced before, tells us that modulo small (orbifold) torsion, by which we mean tensorizing with $R = \Z[\frac{1}{\omega_F^{(2)}}]$, there is a surjection
$$\Homol_2(\PGL_2(\ringO); \thinspace R) \twoheadrightarrow K_2(\ringO) \otimes R.$$
We recall that since $R$ is flat, $\Homol_2(\PGL_2(\ringO); \thinspace R) \cong \Homol_2(\PGL_2(\ringO); \thinspace \Z) \otimes R.$
Now, we state how under some extra hypothesis on $\mathcal{O}$,  our construction can be applied in the setting of Calegari and Venkatesh, i.e. when $N=2$. 
\begin{corollary} \label{lift-to-Steinberg1}
 Assume that $\ringO$ has class number $|{\rm cl}(\ringO)| = 1$, and has a real embedding or is Euclidean.
 Then all non-trivial elements of $K_2(\ringO)$ can be lifted non-trivially (in the sense of Theorem~\ref{lifting}) to elements in $\Homol_{1}(\GL_{2}(\ringO);\thinspace \St_{2}(F))$ or $\Homol_{2}(\GL_{1}(\ringO);\thinspace \St_{1}(F))$. \\
 Modulo orbifold torsion, the Bieri-Eckmann duality allows non-trivial lifts in $\Homol^{\nu(2)-1} (\GL_2 (\ringO) , R)$ or $\Homol^{\nu(1)-2} (\GL_1 (\ringO) , R)$.
\end{corollary}
\noindent
We defer the proof of this corollary to Section~\ref{CFP}. Now, we focus on the real quadratic case, where for example $\nu(1) = r_1+r_2 -1 = 1$, so the last mentioned module vanishes.
\begin{corollary}\label{GLtoPGL}
Let $F$ be a real quadratic field of class number one.
Then all non-trivial elements of $K_2(\ringO)\otimes R$ can be lifted non-trivially (in the sense of Theorem~\ref{lifting}) to elements into $\Homol^3(\GL_2(\ringO_F); \thinspace R)$ and further to $\Homol^t(\PGL_2(\ringO_F); \thinspace R)$ with $t=2$ or $t=3$.
\end{corollary}
\noindent
We defer also this proof to Section~\ref{CFP}.
Finally, we present the last application which is a direct consequence of the previous corollaries. We have the following:

\begin{corollary}\label{divisibility}
Let $p \geq 7$ be a prime, and let $F$ be a real quadratic field of class number one. Assume that $\Homol^3 (\text{GL}_2 (\mathcal{O}_F ), \mathbb{Q})=0$. If $p$ divides the order of $K_2 (\mathcal{O}_F )$, then $p$ divides the order of $\Homol_2 ( \text{PGL}_2 (\mathcal{O}_F ), \mathbb{Z})_{\text{tors}}.$
\end{corollary}
\noindent
Note that Prestel \cite{Prestel}*{\S 7} has shown that the orbifold primes of $\GL_2$ over a real quadratic ring $\ringO$ are in $\{2, 3, 5\}$, hence the restriction to primes $p\geq 7$. The above corollary provides a sufficient condition for the existence of torsion classes in $\Homol_2 (\mathcal{Y}_0 (1) , R)$, where $R=\mathbb{Z}[\frac{1}{30}]$. Note also that, as we will see later on, we know that $\Homol^4 (\text{GL}_2 (\mathcal{O}_F ), R )=0$, thanks to the vanishing theorem in ~\cite{ChurchFarbPutman}*{Theorem C}. However, we are not aware of existing conjectures, analogue to the case of $\text{GL}_2 (\mathbb{Z})$, which could predict the rank of $\Homol^i (\text{GL}_2 (\mathcal{O}_F ), \mathbb{Q})$ strictly under the bound $i=\text{vdc}(2)=4$. 

\subsubsection{Real quadratic ring examples}
Let us consider the case of $r_1 = 2$ real embeddings, $r_2=0$ pairs of complex embeddings, so $\ringO$ is the ring of algebraic integers in $\rationals(\sqrt{m})$ with $1 < m$ square-free,
and $\nu(2) = \vcd(\GL_2(\ringO)) = 4$.
With a method of de Jeu, described in Section~\ref{detecting-odd-torsion}, the authors have computed the order of $K_2(\ringO)$ for $1 < m < 2000$ square-free, in those cases where $\ringO$ is of class number $1$. Prestel \cite{Prestel}*{\S 7} has shown that the orbifold primes of $\GL_2$ over a real quadratic ring $\ringO$ are in $\{2, 3, 5\}$. So to have Bieri--Eckmann duality (isomorphisms in Recollection~\ref{Bieri-Eckmann duality}), we only list torsion at primes $p > 5$.
The outcomes are displayed in Tables~\ref{smaller100} and~\ref{larger100}.

\begin{table}[b]
\begin{mdframed}
\caption{Square-free integers $m$ with $1 < m < 100$ such that $\ringO_{\rationals(\sqrt{m})}$ is a principal ideal domain (PID), and such that there is a prime number $p > 5$ (hence a non-orbifold prime) in the order of $K_2(\ringO_{\rationals(\sqrt{m})})$. When a square-free integer $m$ is not listed, it can be either because $\ringO_{\rationals(\sqrt{m})}$ is not a PID, or because $K_2(\ringO_{\rationals(\sqrt{m})})$ contains no torsion of order $p > 5$.}
\label{smaller100}
$$
\footnotesize
\begin{array}{|l|c|c|c|c|c|c|c|c|c|c|c|c|c|c|c|c|c|c|c|c|c|}
 \hline
m & 11 & 19 & 22 & 38 & 43 & 46 & 47 & 53 & 57 & 59 & 61 & 62 & 67 & 71 & 73 & 83 & 86 & 89 & 94 & 97 \\ \hline
p & 7  & 19 & 23 & 41 &  7 & 37 &  7 & 7  & 7  & 17 & 11 & 7  & 41 & 29 & 11 & 43 & 31 & 13 & 53 & 17 \\ \hline
\end{array}
$$
We note that among the examples in this table, the rings of integers in $\rationals(\sqrt{11})$, $\rationals(\sqrt{19})$, $\rationals(\sqrt{57})$ and $\rationals(\sqrt{73})$ are the only norm-Euclidean ones, so they are covered by Lee and Szczarba's result~\cite{LeeSzczarba}*{corollary to theorem 4.1} that for $n\geq 2$
and $\ringO$ Euclidean with multiplicative norm, $\Homol_0(\GL_n(\ringO),\St_n(\ringO)) = 0.$
We observe that there is at most one prime number $p > 5$, in the order of $K_2(\ringO_{\rationals(\sqrt{m})})$ for $\ringO_{\rationals(\sqrt{m})}$ a principal ideal domain with $1 < m < 134$, and in such cases the multiplicity of $p$ is $1$.
This phenomenon ends at $m = 134$, where we observe prime factors $7$ and $43$. And the limit on the multiplicity ends at $m = 193$, where a factor $7^2$ can be observed.
\\
Note also that all prime numbers $7 \leq p \leq 43$ appear in this table, which indicates that there might be some kind of uniform distribution of non-orbifold primes over $\{|K_2(\ringO_{\rationals(\sqrt{m})})|\}_{m \text{ square-free}}.$
\end{mdframed}
\end{table}

\begin{table}
\begin{mdframed}
\caption{Square-free integers $m$ with $100 < m \leq 2000$ such that $\ringO_{\rationals(\sqrt{m})}$ is a principal ideal domain, and such that the order of $K_2(\ringO_{\rationals(\sqrt{m})})$ contains a prime number $p \geq m$.}
\label{larger100}
\footnotesize
$$
\begin{array}{|l|c|c|c|c|c|c|c|c|c|c|c|c|c|c|c|c|c|c|c|c|c|}
 \hline m & 107 & 118 & 163 & 166 & 214 & 251 & 283 & 302 & 331 & 334 & 419 & 422 & 478 & 487 & 563 & 571\\ \hline
 p       & 197 & 277 & 467 & 503 & 757 & 251 & 1021& 479 & 487 & 719 & 1619 & 1559 & 593 & 601 & 761 & 1097\\ \hline
\end{array}
$$

$$
\begin{array}{|l|c|c|c|c|c|c|c|c|c|c|c|c|c|c|c|c|c|c|c|c|c|}
 \hline m &  643 & 739 & 787 & 811 &  827 &  838 &  862  & 947  & 967  & 1019 & 1046 & 1094 & 1123 & 1163\\ \hline
 p      & 701 & 4831 & 1549 & 5297 & 4153 & 5101 & 1433 & 5231 & 1669 & 6211 & 2113 & 7039 & 1619 & 7001\\ \hline
\end{array}
$$

$$
\begin{array}{|l|c|c|c|c|c|c|c|c|c|c|c|c|c|c|c|c|c|c|c|c|c|}
 \hline m &  1259 & 1303 & 1307 & 1382 & 1427 & 1454 & 1531 & 1543  & 1607  & 1699 & 1718 & 1723 & 1774\\ \hline
 p      & 8377 & 2663 & 2711 & 8753 & 9293 & 1697 & 2801 & 3329 & 2903 & 5507 & 12097& 15667& 8737 \\ \hline
\end{array}
$$

$$
\begin{array}{|l|c|c|c|c|c|c|c|c|c|c|c|c|c|c|c|c|c|c|c|c|c|}
 \hline m  & 1814 & 1822 & 1867 & 1879 & 1931 & 1979 \\ \hline
 p        & 2111 & 2143 & 5689 & 4789 & 15443& 3253 \\ \hline
\end{array}
$$
\normalsize

Similarly to Table~\ref{smaller100}, in this table there is at most one prime number $p \geq m$ in the order of $K_2(\ringO_{\rationals(\sqrt{m})})$ for $\ringO_{\rationals(\sqrt{m})}$ a principal ideal domain with $100 < m < 2000$, and then the multiplicity is $1$.
\end{mdframed}
\end{table}
\subsubsection{Cubic rings}\label{Cubic rings}
Producing explicit examples over a cubic number ring $\ringO$ is difficult, because cubic rings make the computation of the order of $K_2(\ringO)$ much more complicated.
\begin{itemize}
 \item Let us just consider the case of the number field $F$ generated by a root of the polynomial
\mbox{$x^3 - x^2 + 2\cdot x + 12$}, of discriminant $-5^2\cdot 43=-1075$.
An experimental computation of Gangl yields $7 | K_2(\ringO_F)$.
As the class group of $\ringO_F$ is trivial, and there is a real embedding ($r_1 = r_2 = 1, \nu(2) = 5$), we can apply Corollary~\ref{lift-to-Steinberg1} to obtain a lift of the $7$-torsion from $K_2(\ringO_F)$ to $\Homol_{1}(\GL_{2}(\ringO_F);\thinspace \St_{2}(F))$.
Applying Bieri-Eckmann duality then transports our lift into $\Homol^{4}(\GL_{2}(\ringO_F);\thinspace \Z)$.

\item As another example, let us consider the totally real field $F'$ generated by a root of the polynomial
\mbox{$x^3 - 6\cdot x - 2$}, of discriminant $756$ and with $r_1 = 3, r_2 = 0, \nu(2) =7$.
Gangl experimentally finds $13$-torsion in $K_2(\ringO_{F'})$.
Also here we have trivial class group and real embeddings, so we apply Corollary~\ref{lift-to-Steinberg1} to lift of the $13$-torsion from $K_2(\ringO_{F'})$ to $\Homol_{1}(\GL_{2}(\ringO_{F'});\thinspace \St_{2}(F'))$.
Applying Bieri-Eckmann duality then transports our lift into $\Homol^{6}(\GL_{2}(\ringO_{F'});\thinspace \Z)$.
\end{itemize}

\subsection*{Application of the theorem at $N = 1$.}
Denoting by cl$(\ringO)$ the class group of $\ringO$, Putman and Studenmund~\cite{PutmanStudenmund}*{theorems A and B} have established that
$$\dim_\Q \Homol^{\nu(n)}(\GL_n(\ringO); \thinspace \Q) \begin{cases}
                                                         = 0, \medspace n \text{ even, } r_1+r_2 \geq n \text{ and } \ringO^\times \text{ contains an element of norm $-1$}; \\
                                                         \geq (|\text{cl}(\ringO)| -1)^{n-1}, \medspace n \text{ odd, or } \ringO^\times \text{ not containing  elements of norm $-1$}. \\
                                                        \end{cases}
 $$
Let $ \ringO^\times$ contain an element of norm $-1$, let $r_1+r_2 \geq 2$.
Then we get $\dim_\Q \Homol^{\nu(2)}(\GL_2(\ringO); \thinspace \Q) = 0$, and applying Theorem~\ref{lifting} to $N=1$, we conclude that for $p$ not an orbifold prime, all $p$-torsion elements of $K_1(\ringO)$
can be lifted to $\Homol^{\nu(1)-1}(\GL_1(\ringO); \Z) = \Homol^{r_1+r_2-2}(\ringO^\times; \Z)$ or to $p$-torsion elements in $\Homol^{\nu(2)}(\GL_2(\ringO); \Z)$.

\subsubsection*{Structure of the paper}
The proof of Theorem~\ref{lifting} is given in Section~\ref{Steinberg_connection}.
An extension to large torsion of Church, Farb and Putman's theorem on the vanishing of top-degree cohomology is described in Section~\ref{CFP}. We use it to prove Corollaries~\ref{lift-to-Steinberg1} and~\ref{GLtoPGL}.
A method for computing the order of $K_2$ of real quadratic rings is given in Section~\ref{detecting-odd-torsion}.

\subsubsection*{Acknowledgements} We are grateful to Philippe Elbaz-Vincent (University of Grenoble Alpes) for sharing his insights on the topic.  
We are indebted to Rob de Jeu (VU Amsterdam) for providing the theorem for computing the order of $K_2(\ringO)$ in Section~\ref{detecting-odd-torsion}.
We would like to heartily thank Herbert Gangl (Durham University) for computing the examples in Section~\ref{Cubic rings}.
And we are very grateful to Peter Patzt (University of Oklahoma) for a careful read of the manuscript and helpful suggestions for its improvement.
We would like to acknowledge financial support by the ANR grant MELODIA (ANR-20-CE40-0013).

\section{The connection via Steinberg homology}\label{Steinberg_connection}

In this section, we give the proof of the theorem stated in the Introduction.

\subsubsection*{Notation and Terminology}
Let $F$ be a number field, denote by $\text{Gal}(\overline{\mathbb{Q}} / F)$ the group of automorphisms of $\overline{\mathbb{Q}}$ which fixes $F$. We denote by $\omega_F$ the order of the group of roots of unity $\mu_F$ and by $\omega_F^{(2)}$ be the order of the group of Galois invariant elements in the 2nd Tate twist $\mathbb{Q}/\mathbb{Z} (2)$, i.e. $(\mathbb{Q}/\mathbb{Z} \otimes \mu^{\otimes 2}_\infty)^{\text{Gal}(\overline{\mathbb{Q}} / F )}$. For any commutative ring with unity $A$, denote by $K_2^M (A) $ Milnor's second $K$-group and denote by $K_i^Q (A)$ (or simply $K_i (A)$) Quillen's $K$-groups. It is well-known that the groups $K_2^M (A)$ and $K_2 (A)$ are isomorphic when $A$ is a field, or more generally when $A$ has "enough" units (see ~\cite{EVMS}*{sec. 1} and ~\cite{VDK}*{sec. 8.5 page 512}. Generally speaking, Theorem~\ref{lifting} gives a way to construct, starting from elements in the $N$-th $K$-group of a ring $\ringO$ of algebraic numbers, a non-zero element in the Steinberg homology groups, and later on in the applications in the (co)homology of $\text{PGL}_2 (\mathcal{O})$ (modulo small torsion). We refer sometimes to these constructed elements as lifts of the original element in the $K$ group. As it will be clear from the proof, our construction depends on chasing elements in a diagram relating Quillen's $Q$-construction of $K$-theory and Steinberg homology. The term lift comes from the fact that the set-theoretical map $\varphi$ is defined via a set-theoretical map $\psi$ which is a categorical lift in the category of sets because it is defined by taking pre-images.

For the purpose of the proof, we simplify the notation
$\Homol_m(\GL_j(\ringO); \thinspace \St_{j}(F))$,
over the ring of algebraic integers $\ringO$ in a number field $F$ and the associated Steinberg module $\St_{j}(F)$, to
$\widetilde{\Homol_m(\GL_j(\ringO))}$. We will keep this notation in the following section.

\begin{proofof}{Theorem \ref{lifting}}
First, we make use of the Hurewicz homomorphism
$$\text{Hu}_N : K_N(\ringO) := \pi_{N+1}(BQ) \to \Homol_{N+1}(BQ;\thinspace \Z),$$
where $BQ$ is the classifying space from Quillen's $Q$-construction of algebraic $K$-Theory.
Quillen's array of long exact sequences~\cite{Quillen} connects the $Q$-construction with the Steinberg homology $\widetilde{\Homol_*}$ of the general linear group -- for all $j \in \N$, there is an exact sequence
$$\hdots \to \Homol_m(BQ_{j}) \to \widetilde{\Homol_{m-j}(\GL_j(\ringO))} \to \Homol_{m-1}(BQ_{j-1}) \to \Homol_{m-1}(BQ_{j}) \to \widetilde{\Homol_{m-j-1}(\GL_j(\ringO))} \to \hdots,$$
where the Steinberg homology groups $\widetilde{\Homol_*}$ have coefficients in the Steinberg module of $\GL_j(\ringO)$, and the ordinary homology groups $\Homol_*$ have trivial $\Z$-coefficients.
\\
To link with the $K$-group $K_N(\ringO)$, we connect the first $(N+3)$ of these exact sequences to the following diagram, where the top row is the sequence for $j=1$, below it are the ones for $j=2$, and so on until finally the one for $j=N+3$.
\begin{center}
\hspace*{-9mm}
\begin{tikzpicture}[descr/.style={fill=white,inner sep=1.5pt}]
        \matrix (m) [
            matrix of math nodes,
            row sep=1em,
            column sep=2.5em,
            text height=1.99ex, text depth=0.75ex
        ]
        {
          &  & 0 = \Homol_{N+1}(BQ_0) & \Homol_{N+1}(BQ_1) & \widetilde{\Homol_{N}(\GL_1 (\ringO))}  & \hdots \\
        \hdots  & \widetilde{\Homol_N(\GL_2(\ringO))} & \Homol_{N+1}(BQ_1) & \Homol_{N+1}(BQ_2) & \widetilde{\Homol_{N-1}(\GL_2(\ringO))}  & \hdots \\
        \hdots  & \widetilde{\Homol_{N-1}(\GL_3(\ringO))} & \Homol_{N+1}(BQ_2) & \Homol_{N+1}(BQ_3) & \widetilde{\Homol_{N-2}(\GL_3(\ringO))}  & \hdots \\
  \hdots  & \hdots  & \hdots  & \hdots & \hdots  & \hdots \\
        \hdots  & \widetilde{\Homol_2(\GL_N(\ringO))} & \Homol_{N+1}(BQ_{N-1}) & \Homol_{N+1}(BQ_N) & \widetilde{\Homol_{1}(\GL_N(\ringO))}  & \hdots \\
        \hdots  & \widetilde{\Homol_{1}(\GL_{N+1}(\ringO))} & \Homol_{N+1}(BQ_N) & \Homol_{N+1}(BQ_{N+1}) & \widetilde{\Homol_{0}(\GL_{N+1}(\ringO))}  & \hdots \\
        \hdots  & \widetilde{\Homol_0(\GL_{N+2}(\ringO))} & \Homol_{N+1}(BQ_{N+1}) & \Homol_{N+1}(BQ_{N+2}) & 0  & \hdots \\
        \hdots  & 0 & \Homol_{N+1}(BQ_{N+2}) & \Homol_{N+1}(BQ_{N+3})  & 0  & \hdots \\
        };

        \path[overlay,->, font=\scriptsize,>=latex]
        (m-1-3) edge (m-1-4)
        (m-1-4) edge (m-1-5)
        (m-1-5) edge (m-1-6)

        (m-2-1) edge (m-2-2)
        (m-2-2) edge (m-2-3)
        (m-2-3) edge (m-2-4)
        (m-2-4) edge (m-2-5)
        (m-2-5) edge (m-2-6)

        (m-3-1) edge (m-3-2)
        (m-3-2) edge (m-3-3)
        (m-3-3) edge (m-3-4)
        (m-3-4) edge (m-3-5)
        (m-3-5) edge (m-3-6)

        (m-4-1) edge (m-4-2)
        (m-4-2) edge (m-4-3)
        (m-4-3) edge (m-4-4)
        (m-4-4) edge (m-4-5)
        (m-4-5) edge (m-4-6)

        (m-5-1) edge (m-5-2)
        (m-5-2) edge (m-5-3)
        (m-5-3) edge (m-5-4)
        (m-5-4) edge (m-5-5)
        (m-5-5) edge (m-5-6)

        (m-6-1) edge (m-6-2)
        (m-6-2) edge (m-6-3)
        (m-6-3) edge (m-6-4)
        (m-6-4) edge (m-6-5)
        (m-6-5) edge (m-6-6)

        (m-7-1) edge (m-7-2)
        (m-7-2) edge (m-7-3)
        (m-7-3) edge (m-7-4)
        (m-7-4) edge (m-7-5)
        (m-7-5) edge (m-7-6)

        (m-8-1) edge (m-8-2)
        (m-8-2) edge (m-8-3)
        (m-8-3) edge (m-8-4)
        (m-8-4) edge (m-8-5)
        (m-8-5) edge (m-8-6)
        ;
\path[overlay,->, font=\normalsize,>=latex] (m-2-3) edge (m-1-4) node[midway, above, yshift=19.3ex] {$=$};
\path[overlay,->, font=\normalsize,>=latex] (m-3-3) edge (m-2-4) node [midway, above, yshift=12.7ex] {$=$};
\path[overlay,->, font=\normalsize,>=latex] (m-4-3) edge (m-3-4) node [midway, above, yshift=6.1ex] {$=$};
\path[overlay,->, font=\normalsize,>=latex] (m-5-3) edge (m-4-4) node[midway, above, yshift=-0.5ex] {$=$};
\path[overlay,->, font=\normalsize,>=latex] (m-6-3) edge (m-5-4) node [midway, above, yshift=-7.1ex] {$=$};
\path[overlay,->, font=\normalsize,>=latex] (m-7-3) edge (m-6-4) node [midway, above, yshift=-13.7ex] {$=$};
\path[overlay,->, font=\normalsize,>=latex] (m-8-3) edge (m-7-4) node [midway, above, yshift=-20.3ex] {$=$};
\end{tikzpicture}
\end{center}
\noindent
where we have exploited that $Q_0$ is equivalent to the trivial category~\cite{Quillen}, and that
$$\widetilde{\Homol_{-m}(\GL_{N+1}(\ringO))} = 0 \text{ for } m > 0.$$ That means, the last map in the second-last row,
\begin{center}$\Homol_{N+1}(BQ_{N+1}) \to \Homol_{N+1}(BQ_{N+2})$, is a surjection,
\end{center}
and we can lift our torsion classes up along it.
The last row shows the beginning of the chain of isomorphisms yielding Quillen's stability
$$\Homol_{N+1}(BQ_{N+2}) \cong \Homol_{N+1}(BQ_{N+3}) \cong \Homol_{N+1}(BQ_{N+4}) \cong \hdots \cong \Homol_{N+1}(BQ).$$
Now, let $\alpha \in K_N(\ringO)$ be a large torsion element (in the sense of Def. \ref{arlettaz}) or an element of infinite order. Via the above chain of isomorphisms, the image of our element $\alpha$ via the Hurewicz homomorphism arrives in $\Homol_{N+1}(BQ_{N+2})$, and as announced, we lift it up to $\Homol_{N+1}(BQ_{N+1})$.
Our lift of $\text{Hu}_N (\alpha)$, a non-zero class in $\Homol_{N+1}(BQ_{N+1})$, is now found again in the third-last row, and \begin{itemize}                                                                                                                        \item either it yields a non-zero class in  $\widetilde{\Homol_{0}(\GL_{N+1}(\ringO))}$,
\item or it is in the kernel of the map going there, and we can lift it again, this time against the homomorphism $\Homol_{N+1}(BQ_N) \to \Homol_{N+1}(BQ_{N+1}).$                                                                                                                         \end{itemize}
We continue this game of climbing up the ladder and lifting, until we have shown that $\alpha$ can be lifted into $\widetilde{\Homol_{0}(\GL_{N+1}(\ringO))}$, $\widetilde{\Homol_{1}(\GL_{N}(\ringO))}$, $\hdots$, $\widetilde{\Homol_{N-2}(\GL_{3}(\ringO))}$, $\widetilde{\Homol_{N-1}(\GL_{2}(\ringO))}$ or $\widetilde{\Homol_{N}(\GL_1(\ringO))}$.\\
It is now straightforward to deduce from the construction described above the definition and properties of the set-theoretical map $\varphi$. 
\end{proofof}

\section{Vanishing of top-degree cohomology}
\label{CFP}
\noindent
Lee and Szczarba~\cite{LeeSzczarba}*{corollary to theorem 4.1} have shown that for $n \geq 2$ and $\ringO$ Euclidean with multiplicative norm, $\widetilde{\Homol_{0}(\GL_{n}(\ringO))} =0$. Church, Farb and Putman have extended this to $\ringO$ of class number $1$, with a real embedding or Euclidean -- in this section, we extend their result to large torsion, and prove the corollaries from the Introduction.\\
The ``Vanishing Theorem'' of Church, Farb and Putman~\cite{ChurchFarbPutman}*{Theorem C} extends as follows to large torsion.
\begin{theorem}[Consequence of Church, Farb and Putman's work] \label{vanishing-theorem}
 Let $\ringO_F$ be the ring of integers in an algebraic number field $F$ with $|{\rm cl}(\ringO_F)| = 1$. Suppose that $F$ has a real embedding or that $\ringO_F$ is Euclidean. Let $n \geq 2$, and let $P$ be the product of all orbifold primes in $\GL_n(\ringO_F)$. Then
$$\Homol^{\vcd(\SL_n(\ringO_F))}\left(\SL_n(\ringO_F); \thinspace \Z\left[\frac{1}{P}\right]\right) =
\Homol^{\nu(n)}\left(\GL_n(\ringO_F); \thinspace \Z\left[\frac{1}{P}\right]\right) = 0.$$
\end{theorem}
\begin{proofof}{Theorem~\ref{vanishing-theorem}}
 The quoted Theorem C (``Vanishing Theorem'') also makes the analogous statement for twisted coefficient systems arising from rational representations of the algebraic group $\GL_n$ ; we are only interested in its statement for untwisted coefficients $\rationals$, which we extend to untwisted coefficients $\Z\left[\frac{1}{P}\right]$.
 \\
 Apart from inverting the prime $2$ for showing that coinvariants vanish, which we have done already because $2$ divides $P$, all the arguments of Church, Farb and Putman~\cite{ChurchFarbPutman}*{section 4, proof of Theorem C} pass over to our setting, we just have to notice that Bieri-Eckmann duality extends to $\Z\left[\frac{1}{P}\right]$ -- see Recollection~\ref{Bieri-Eckmann duality} above.
 In fact, they prove that the $\SL_n(\ringO_F)$-coinvariants of the Steinberg module vanish, and hence
 $$\widetilde{\Homol_{0}(\GL_n(\ringO_F))} = \widetilde{\Homol_{0}(\SL_n(\ringO_F))} = 0.$$
\end{proofof}

\begin{proofof}{Corollary~\ref{lift-to-Steinberg1}}
It is known that $K_2(\ringO)$ is a finite Abelian group~(see Quillen's paper \cite{Quillen} for the finite generation, and Borel's proposition 12.2 \cite{Borel74} for the vanishing rank).
Applying Theorem~\ref{lifting} in this setting to $N = 2$, and using the consequence of Theorem~\ref{vanishing-theorem} that $\widetilde{\Homol_0(\GL_3(\ringO))}= 0$, we obtain that all elements from $K_2(\ringO)$ (because $p > \frac{N+1}{2}$ for all prime numbers $p$) can be lifted to $\widetilde{\Homol_{1}(\GL_{2}(\ringO))}$ or $\widetilde{\Homol_{2}(\GL_{1}(\ringO))}$.
\end{proofof}

\begin{proofof}{Corollary~\ref{GLtoPGL}}
As $\ringO_F$ is real quadratic, $\nu(1) = r_1+r_2 -1 = 1$, so
$\widetilde{\Homol_{2}(\GL_{1}(\ringO_F))}
\cong \Homol^{\nu(1)-2} (GL_1 (\ringO_F) , R)$ vanishes.
To the lift from Corollary~\ref{lift-to-Steinberg1}, we apply the isomorphism modulo orbifold primes
$$\widetilde{\Homol_{1}(\GL_{2}(\ringO_F))} \cong \Homol^{\nu(2)-1}(\GL_2(\ringO_F); \thinspace \Z)$$
from Recollection~\ref{Bieri-Eckmann duality}.
We get $\nu(2) = 4$, so the classes lifted from elements not containing orbifold torsion arrive in $\Homol^{3}(\GL_2(\ringO_F);\thinspace R)$.
As we have assumed $\ringO_F$ to be a principal ideal domain, the sequence
\begin{equation}
 1 \to \ringO_F^\times \to \GL_2(\ringO_F) \to \PGL_2(\ringO_F) \to 1
 \label{group-extension}
\end{equation}
is exact. Moreover, it is a central extension, because the action of $\PGL_2(\ringO_F)$ on $\ringO_F^\times$  by conjugation is trivial. Therefore, in its Lyndon--Hochschild--Serre spectral sequence with trivial $R$-coefficients, the action of $\PGL_2(\ringO_F)$ on $\Homol^t(\ringO_F^\times; R)$ is trivial.
This will allow us to further lift our lifted elements from $\Homol^{\nu(2)-1}(\GL_2(\ringO_F); R)$ into $\Homol^*(\PGL_2(\ringO_F); \thinspace R)$.
Note that the primes appearing in the factorization of the order of the torsion part of $\mathcal{O}^\times$ are invertible in $R$. Then, as Dirichlet's unit theorem provides the group structure
$\ringO^\times \cong \Z^{r_1+r_2-1}\oplus \mu(F)$, we get
$\Homol^t(\ringO_F^\times; R) \cong R^{n(\ringO_F, t)}$ with $n(\ringO_F, t) = $\scriptsize$ \begin{pmatrix}
                                            r_1+r_2-1 \\ t
                                           \end{pmatrix}$ \normalsize
as the cohomology of the $(r_1+r_2-1)$-torus.
Now we use again our restrictive hypothesis that $F$ is real quadratic,
so  $n(\ringO_F, 1) = 1$, and for $t\geq 2$, $n(\ringO_F, t) = 0$, whence $$\Homol^t(\ringO_F^\times; R) \cong  \begin{cases}
                                            R, & t\in \{0,1\}
                                            \\ 0, & t \geq 2.
                                           \end{cases}$$ \normalsize
The above yields the following $E_2$-page of the Lyndon--Hochschild--Serre spectral sequence of the group extension (\ref{group-extension}), using the abbreviation $Q := \PGL_2(\ringO_F)$ :

\hspace*{-8mm}
\begin{tikzpicture}[descr/.style={fill=white,inner sep=1.5pt}]
\matrix (m) [matrix of math nodes,
    nodes in empty cells,nodes={minimum width=5ex,
    minimum height=5ex,outer sep=-5pt},
    column sep=1ex,row sep=1ex]{
    &&&&&&\\
          q\geq 2     &&  0   &  0   &  0   & 0 &\\
          q=1     &&  \Homol^0(Q; \thinspace \Homol^1(\ringO_F^\times))   &  \Homol^1(Q; \thinspace \Homol^1(\ringO_F^\times))    & \Homol^2(Q; \thinspace \Homol^1(\ringO_F^\times))   & \Homol^3(Q; \thinspace \Homol^1(\ringO_F^\times)) & \\
          q=0     &&  \Homol^0(Q; \thinspace \Homol^0(\ringO_F^\times))   &  \Homol^1(Q; \thinspace \Homol^0(\ringO_F^\times))    & \Homol^2(Q; \thinspace \Homol^0(\ringO_F^\times))   & \Homol^3(Q; \thinspace \Homol^0(\ringO_F^\times)) & \\
    \quad\strut && _{p=0}  &  _{p=1}  &  _{p=2}  & _{p=3} \strut & \\};
\draw[thick] (m-1-2.east) -- (m-5-2.east) ;
\draw[thick] (m-5-2.north) -- (m-5-7.north) ;
\end{tikzpicture}\\
\noindent
So the group $\Homol^{3}(\GL_2(\ringO_F);\thinspace R)$ is an extension of $\Homol^{2}(\PGL_2(\ringO_F);\thinspace R)$ by $\Homol^{3}(\PGL_2(\ringO_F);\thinspace R)$.\end{proofof}

\begin{proofof}{Corollary~\ref{divisibility}}
First, note that Prestel \cite{Prestel}*{\S 7} has shown that the orbifold primes of $\GL_2$ over a real quadratic ring $\ringO$ are in $\{2, 3, 5\}$, hence the restriction to primes $p\geq 7$. Let $g\in K_2 (\mathcal{O}_F )$ be a non-zero element of order $p$. Corollary 1.8 ensures us that starting from such $g\in K_2 (\mathcal{O}_F )$, we can lift it to a non-trivial element $g' \in \Homol_3 (BQ_2 )$, i.e. $\pi_2 (g' )= \text{Hu}_2 (g)$, such that $p$ divides $\text{ord} (g')$. Now, note that as observed in the proof of Corollary 1.8, we have $\Homol_{0} (\GL_{3}(\mathcal{O}), \St_{3}(F))=\Homol_{2} (\GL_{1}(\mathcal{O}), \St_{1}(F))=0$. \\
In particular, the group homomorphism $\phi_2 : H_3 (BQ_2 ) \rightarrow \Homol_{1} (\GL_{2}(\mathcal{O}), \St_{2}(F))$ is injective. Thanks to the Bieri-Eckman duality and under the hypothesis that $\Homol^3 (\text{GL}_2 (\mathcal{O}_F ), \mathbb{Q})=0$, we have that $g'' = \phi_2 (g')$ is a non-zero torsion element of order divisible by $p$. 
We keep denoting by $g''$ the corresponding non-zero element in $\Homol^3 (\text{GL}_2 (\mathcal{O}_F ), R)$, where $R=\mathbb{Z}[\frac{1}{30}]$. Since $\Homol^3 (\text{GL}_2 (\mathcal{O}_F ), R)$ is an extension of $\Homol^2 (\text{PGL}_2 (\mathcal{O}), R)$ by $\Homol^3 (\text{PGL}_2 (\mathcal{O}), R)$ (see the proof of the previous corollary), we define $g'''$ as the projection of $g''$ in $\Homol^2 (\text{PGL}_2 (\mathcal{O}), R)$. By the universal coefficient theorem, the torsion element $g'''$ has to come from a torsion element of $\text{Ext}^1_\mathbb{Z} (\Homol_1 (\text{PGL}_2 (\mathcal{O}), \mathbb{Z} ), R)$. However, this group is trivial since the primes dividing the order of the abelianization of $\text{PGL}_2 (\mathcal{O}_F )$ are 2 and 3 (see ~\cite{Mir}), and we are inverting those in $R$. We deduce from this that the element $g''$ has to come from a torsion element $h$ in $\Homol^3 (\text{PGL}_2 (\mathcal{O}), R)$ and the order of $h$ has to be divisible by $p$. Again by the universal coefficient theorem, we finally get a torsion element in $\Homol_2 (\text{PGL}_2 (\mathcal{O}), R)$ whose order is divisible by $p$ and this completes the proof. 
\end{proofof}

\section{\texorpdfstring{Computing the order of $K_2$ of real quadratic rings}{Computing the order of K2 of real quadratic rings}}
\label{detecting-odd-torsion}
\noindent
 In order to experimentally detect odd torsion in $K_2$ of real quadratic rings, we
use the observation at the beginning of~\cite{BelabasGangl} that~$ |K_2(\ringO)| $
can be computed from the value of its Dedekind~$ \zeta $-function
$ \zeta_F $ at $ s = -1 $ whenever the field of fractions~$ F $ of~$ \ringO $
is totally real and abelian. We give the argument for $ F $
real quadratic.

\begin{theorem}\label{zeta}
Let  $ F $ be real quadratic number field, $ F \neq \rationals(\sqrt2) $ and~$ F \neq \rationals(\sqrt5)$.
Then $|K_2(\ringO)| = 24\zeta_F(-1)$.
\end{theorem}
\begin{proofof}{Theorem~\ref{zeta}}
By~\cite{WeiKbook}*{Theorem~VI.8.8}
or~\cite{BurnsEtAl}*{Corollary~2.6} we have
$$\zeta_F(-1)= 2^2\frac{|K_2(\ringO)|}{|K_3(\ringO)|}.$$
If $ F $ is any
number field, then one has~$K_3(F) = K_3(\ringO)$ (see \cite{WeiKbook}*{Theorem~V.6.8}),
and this is known to be finite if (and only if)~$ F $ is totally real
by results of Quillen~\cite{Quillen} and Borel~\cite{Borel74}.
The Milnor $K$-group $K_3^M(F)$ injects into $K_3(F)$
for any field~$ F $
according to \cite{WeiKbook}*{Proposition~VI.4.3.2}.
For~$ F $ real quadratic the order of~$ K_3^M(F) $ is~$ 2^2 $ by
\cite{Ba-Ta} (or see \cite{WeiKbook}*{Examples~III.7.2(d)}), so if we let~$K_3^\text{ind}(F) = K_3(F) / K_3^M(F) $
then we have~$ |K_3(F)| = 2^2 |K_3^\text{ind}(F)| $, and
$ |K_2(\ringO)| = \zeta_F(-1) \cdot |K_3^\text{ind}(F)| $.\\
Because~$ K_3^\text{ind}(F) $ is finite for our~$ F $, we can compute its order using \cite{Levine}*{Cor.~4.6}.
More specifically, if $ p $ is a prime number then
the $ p $-primary part of $ K_3^\text{ind}(F) $ is isomorphic
to~$ H^0(F, \rationals_p / \Z_p(2)) $, where $ (2) $ denotes
the twist by the $ p $-cyclotomic character.
If the real quadratic field $F$ is not contained in $\rationals(\mu_{p^n})$ for
some~$n \in \N$, then Gal$(\overline{\rationals}/F)$ surjects onto Gal$(\rationals(\mu_{p^n})/\rationals)$,
and the $ p $-primary part is the same as for $\rationals$ instead
of~$F$, i.e., isomorphic to $ \Z/8\Z $ for $ p = 2 $, to $ \Z/3\Z $
for~$ p =3 $, and trivial otherwise.
If $ F \subseteq \rationals(\mu_{p^n}) $ for some $ n \ge 1  $,
and $ p \ge 5 $, then the image of Gal$(F/\rationals) $ consists
of the squares in Gal$(\rationals(\mu_{p^\infty}/\rationals)) \simeq \Z_p^* $.
In this case, if $ p > 5 $ then there is an~$ a $ in~$ \Z_p^* $ such that $ a^4-1 $
is also in $ \Z_p^* $, hence the $ p $-primary part of~$ K_3^\text{ind}(F) $
is trivial, but for $ p = 5 $ we have $ F = \rationals(\sqrt5) $
and the 5-primary part is isomorphic to~$ \Z/5\Z $.
The remaining case is for $ p = 2 $, and $ F = \rationals(\sqrt2) \subset \rationals(\mu_8) $.
Here~Gal$(F/\rationals) $ has image~$ \{\pm1\}\cdot(1+8\Z_2) $ in~Gal$(\rationals(\mu_{2^\infty}) / \rationals) \simeq \Z_2^* $,
so that the 2-primary part of~$ K_3^\text{ind}(\rationals(\sqrt2)) $ is
cyclic of order~16.\\
Therefore $ K_3^\text{ind}(F) $ is cyclic, of order~$ 24 $
unless~$ F = \rationals(\sqrt{2}) $, where the order is~48, or~$ F = \rationals(\sqrt{5}) $, where the order
is~120.
So~$ |K_2(\ringO)| = 24 \zeta_F(-1)$
unless~$ F =  \rationals(\sqrt{2}) $, where we
have to multiply this by~$ 2 $, or~$ F = \rationals(\sqrt{5}) $,
where we have to multiply it by~$ 5 $.
\end{proofof}
\noindent
Computing~$24\zeta_F(-1)$ in Pari/GP~\cite{PariGP}, and adjusting
this if~$ F = \rationals(\sqrt2) $ or~$ \rationals(\sqrt5) $, then yields~$|K_2(\ringO)|$.

\end{document}